\documentclass[12pt,leqno]{amsart} 
\newtheorem{theorem}{Theorem}[section]
\newtheorem{definition}{Definition}[section]
\newtheorem{lemma}{Lemma}[section]

\newcommand{\be}{\begin{equation}}
\newcommand{\ee}{\end{equation}}
\newcommand{\bea}{\begin{eqnarray}}
\newcommand{\eea}{\end{eqnarray}}
\newcommand{\beb}{\begin{eqnarray*}}
\newcommand{\eeb}{\end{eqnarray*}}
\usepackage{amssymb,amsfonts,amsthm,setspace,indentfirst,mathrsfs}
\usepackage{minibox}
\numberwithin{equation}{section}

\author{Absos Ali Shaikh}
\address{Department of Mathematics,
Aligarh Muslim University,
Aligarh, Uttar Pradesh, INDIA}
\email{aask2003@yahoo.co.in, aashaikh@math.buruniv.ac.in }
\author{Ali Akbar}
\address{Department of Mathematics,
Rampurhat College,
Birbhum, West Bengal 731224, INDIA}
\email{aliakbar.akbar@rediffmail.com }
\begin{document}

\title[On Locally $\phi $-semisymmetric Kenmotsu Manifolds]{On Locally $\phi $-semisymmetric Kenmotsu Manifolds}
\maketitle

\vskip0.1in
\normalsize

 {\bf Abstract} The object of the present paper is to study the locally $\phi $- semisymmetric Kenmotsu manifolds along with the characterization of such notion.

\section{Introduction}

Let $M$ be an n-dimensional, $n\geq 3$, connected smooth Riemannian manifold endowed with the Riemannian metric $g$. Let $ \nabla $, $R$, $S$ and
$r$ be the Levi-Civita connection, curvature tensor, Ricci tensor and the scalar curvature of $M$ respectively. The manifold $M$ is called locally symmetric due to Cartan (\cite{ce1}, \cite{ce2}) if the local geodesic symmetry at $p\in M$ is an isometry, which is equivalent to the fact that
$\nabla R=0$. Generalizing the concept of local symmetry, the notion of semisymmetric manifold was introduced by Cartan \cite{ce3} and fully classified by
Szabo (\cite{sz1}, \cite{sz2}, \cite{sz3}). The manifold $M$ is said to be semisymmetric if
$ (R(U,V).R)(X,Y)Z=0$, for all vector fields $X$, $Y$, $Z$, $U$, $V$ on $M$, where $R(U,V)$ is considered as the derivation of the tensor algebra at each point of $M$.

In 1977 Takahashi \cite{tt} introduced the notion of local $\phi $- symmetry on a Sasakian manifold. A Sasakian manifold is said to be locally $\phi $-symmetric if
\begin{equation} \phi ^2((\nabla_WR)(X,Y)Z)=0 ,\end{equation}
for all horizontal vector fields $X$, $Y$, $Z$, $W$ on $M$  that is all vector fields orthogonal to $\xi$, where $\phi $ is the structure tensor of the manifold $M$. The concept of local $\phi $- symmetry
on various structures and their generalizations or extension are studied in ( \cite{ak}, \cite{aks}, \cite{ats}, \cite{aats}). By extending the notion of semisymmetry and generalizing the concept of local $\phi $- symmetry of Takahashi \cite{tt}, the first author and his coauthor introduced \cite{aaha} the notion of local $\phi $-semisymmetry on a Sasakian  manifold. A Sasakian manifold $M$, $n\geq 3$, is said to be locally $\phi $-semisymmetric if
\begin{equation} \phi ^2((R(U,V).R)(X,Y)Z)=0, \end{equation}
for all horizontal vector fields $X$, $Y$, $Z$, $U$, $V$ on $M$. In the present paper we study locally $\phi $-semisymmetric Kenmotsu manifolds.
The paper is organized as follows:

In section 2 some rudimentary facts and curvature related properties of Kenmotsu manifolds are discussed. In section 3 we study locally $\phi $-semisymmetric Kenmotsu manifolds and obtained the characterization of such notion.

\section{Preliminaries}
Let $M$ be a $(2n+1)$-dimensional connected smooth manifold endowed with an almost contact
metric structure $(\phi ,\xi ,\eta ,g),$ where $\phi $ is a tensor field of type $(1,1),$ $\xi $
is a vector field, $\eta $ is an $1$-form and $g$ is a Riemannian metric on $M$ such that \cite{bde}
\begin{equation}  \phi ^2X=-X+\eta( X)\xi  ,\hspace{10 pt}\eta (\xi )=1.\label{21}\end{equation}
\begin{equation} g(\phi X,\phi Y)=g(X,Y)-\eta (X)\eta (Y)\hspace{10 pt} \label{22}\end{equation}
for all vector fields X, Y on $M$.\\
Then we have \cite{bde}
\begin{equation} \phi \xi =0,\hspace{10 pt}\eta (\phi X)=0,\hspace{10 pt}\eta (X)=g(X,\xi ).\label{23}\end{equation}
\begin{equation} g(\phi X,X)=0 \label{24a}.\end{equation}
\begin{equation} g(\phi X,Y)=-g(X,\phi Y)\label{24b}\end{equation}
for all vector fields X, Y on $M$.\\
If
\begin{equation} (\nabla_X \phi )Y=-g(X, \phi Y)\xi -\eta (Y)\phi X, \label{24} \end{equation}
\begin{equation}\nabla_X \xi=X - \eta (X)\xi,\label{25}\end{equation}
holds on $M$, then it is called a Kenmotsu manifold \cite{kk}.\\
In a Kenmotsu manifold the following relations hold \cite{kk}
\begin{equation} (\nabla_X \eta )Y= g(X, Y) -\eta (X)\eta (Y),\label{26}\end{equation}
\begin{equation}  \eta (R(X, Y)Z)=g(X, Z)\eta (Y)-g(Y, Z)\eta (X), \label{27} \end{equation}
\begin{equation} R(X, Y)\xi =\eta (X)Y-\eta (Y)X, \label{28} \end{equation}
\begin{equation} R(X, \xi )Z =g(X,Z)\xi -\eta (Z)X, \label{29} \end{equation}
\begin{equation} R(X, \xi )\xi  =\eta (X)\xi -X, \label{29d} \end{equation}
\begin{equation} S(X,\xi )=-2n\eta (X), \label{29a}\end{equation}
\begin{equation} (\nabla_WR)(X, Y)\xi = g(X, W)Y -g(Y, W)X- R(X, Y)W,\label{29b}\end{equation}
\begin{equation} (\nabla_WR)(X, \xi )Z = g(X, Z)W -g(W, Z)X- R(X, W)Z,\label{29c}\end{equation}
for all vector fields $X$, $Y$, $Z$ and $W$ on $M$.\\
In a Kenmotsu manifold we also have \cite{kk}
\begin{equation} \begin{array}{rcl} R(X,Y)\phi W&=&g(Y,W)\phi X-g(X,W)\phi Y+g(X,\phi W)Y-g(Y,\phi W)X\\
&+&\phi R(X,Y)W .\end{array}\label{29e}\end{equation}
Applying $\phi $ and using (\ref{21}) we get from (\ref{29e})
\begin{equation} \begin{array}{rcl} \phi R(X,Y)\phi W&=&-g(Y,W)X+g(X,W)Y+g(X,\phi W)\phi Y-g(Y,\phi W)\phi X\\
&-&R(X,Y)W .\end{array}\label{29f}\end{equation}
In view of (\ref{29f}) we obtain from (\ref{29b})
\begin{equation}(\nabla_WR)(X, Y)\xi =g(Y,\phi W)\phi X-g(X,\phi W)\phi Y+\phi R(X,Y)\phi W.\label{29g}\end{equation}

\section{ Locally $\phi $-semisymmetric Kenmotsu Manifolds}

 \begin{definition} A Kenmotsu manifold $M$ is said to be locally $\phi $-semisymmetric if
\begin{equation} \phi ^2((R(U,V).R)(X,Y)Z)=0,  \label{302} \end{equation}
for all horizontal vector fields $X$, $Y$, $Z$, $U$, $V$ on $M$.
\end{definition}
First we suppose that $M$ is a Kenmotsu manifold such that
\begin{equation} \phi ^2((R(U,V).R)(X,Y)\xi )=0,  \label{303}\end{equation}
for all horizontal vector fields $X$, $Y$, $U$ and $V$ on $M$.

Differentiating (\ref{29g}) covariantly with respect to a horizontal vector field $U$, we get
\begin{equation} \begin{array} {rcl} && (\nabla_U\nabla_VR)(X, Y)\xi \\
&=&[g(X,\phi V)g(U,\phi Y)-g(Y,\phi V)g(U,\phi X)+g(\phi U, R(X,Y)\phi V)]\xi \\
&+&\phi (\nabla_U R)(X,Y)\phi V.\end{array}\label{30}\end{equation}
Using (\ref{29e}) we obtain from (\ref{30})
\begin{equation} \begin{array} {rcl}(\nabla_U\nabla_VR)(X, Y)\xi &=&[g(Y,V)g(U,X)-g(X,V)g(U,Y)+g(R(X,Y)V,U)]\xi \\
&+&\phi (\nabla_U R)(X,Y)\phi V.\end{array}\label{31}\end{equation}
Interchanging $U$ and $V$ on (\ref{31}) we get
\begin{equation} \begin{array} {rcl}(\nabla_V\nabla_UR)(X, Y)\xi &=&[g(Y,U)g(V,X)-g(X,U)g(V,Y)+g(R(X,Y)U,V)]\xi \\
&+&\phi (\nabla_V R)(X,Y)\phi U.\end{array}\label{32}\end{equation}
From (\ref{31}) and (\ref{32}) it follows that
\begin{equation} \begin{array} {rcl} (R(U,V).R)(X,Y)\xi  &=&2[g(Y,V)g(U,X)-g(X,V)g(U,Y)-R(X, Y, U, V)]\xi \\
&+&\phi \{(\nabla_U R)(X,Y)\phi V- (\nabla_V R)(X,Y)\phi U\} .\end{array}\label{33}\end{equation}
Again from (\ref{303}) we have
\begin{equation} (R(U,V).R)(X,Y)\xi =0,  \label{34}\end{equation}
From (\ref{33}) and (\ref{34}) we have
\begin{equation} \begin{array} {rcl}&& 2[g(Y,V)g(U,X)-g(X,V)g(U,Y)-R(X, Y, U, V)]\xi \\
&+&\phi \{(\nabla_U R)(X,Y)\phi V- (\nabla_V R)(X,Y)\phi U\} \\
&=& 0 .\end{array}\label{35}\end{equation}
Applying $\phi $ on (\ref{35}) and using (\ref{29e}), (\ref{29g}) and (\ref{23}) we get
\begin{equation} (\nabla_U R)(X,Y)\phi V- (\nabla_V R)(X,Y)\phi U =0.\label{36} \end{equation}
In view of (\ref{35}) and (\ref{36}) we get
\begin{equation}  R(X, Y, U, V)=g(Y,V)g(U,Y)-g(X,V)g(U,Y) ,\label{37}\end{equation}
\begin{equation}  R(X, Y, U, V)=-\{ g(X,V)g( U,Y)-g(Y,V)g(U,X)\} ,\label{37i}\end{equation}
for all horizontal vector fields $X$, $Y$, $U$ and $V$ on $M$. Hence M is of constant $\phi $-holomorphic sectional curvature -1
and hence of constant curvature -1. This leads to the following:
\begin{theorem} If a Kenmotsu manifold $M$ satisfies the condition $ \phi ^2((R(U,V).R)(X,Y)\xi)=0 $, for all horizontal
vector fields $X$, $Y$, $Z$, $U$ and $V$ on $M$, then $M$ is a manifold of constant curvature -1.
\end{theorem}
We consider a Kenmotsu manifold which is locally $\phi $-semisymmetric. Then from (\ref{302}) we have
\begin{equation} (R(U,V).R)(X,Y)Z=g((R(U,V).R)(X,Y)Z,\xi )\xi , \label{38} \end{equation}
from which we get
\begin{equation} (R(U,V).R)(X,Y)Z=-g((R(U,V).R)(X,Y)\xi, Z )\xi  \label{39} \end{equation}
for all horizontal vector fields $X$, $Y$, $Z$, $U$, $V$ on $M$.

Now taking inner product on both side of (\ref{33}) with a horizontal vector field $Z$, we obtain
\begin{equation} g((R(U,V).R)(X,Y)\xi, Z) =g(\phi (\nabla_U R)(X,Y)\phi V,Z)- g(\phi (\nabla_V R)(X,Y)\phi U,Z) .\label{39a}\end{equation}
Using (\ref{24b}) and (\ref{39}) we get from (\ref{39a})
\begin{equation} (R(U,V).R)(X,Y)Z=[g((\nabla_U R)(X,Y)\phi V,\phi Z)- g((\nabla_V R)(X,Y)\phi U,\phi Z)]\xi  \label{39b} \end{equation}
Differentiating (\ref{29e}) covariantly with respect to a horizontal vector field $V$, we get
\begin{equation} \begin{array} {rcl} && (\nabla_VR)(X, Y)\phi Z \\
 &=&[-g(Y,Z)g(V,\phi X)+g(X,Z)g(V,\phi Y)-g(V, R(X,Y)Z)]\xi \\
&+&\phi (\nabla_V R)(X,Y)Z.\end{array}\label{39c}\end{equation}
Taking inner product on both sides of (\ref{39c}) with a horizontal vector field $U$, we obtain
\begin{equation} g\{(\nabla_VR)(X, Y)\phi Z,U\} =g\{\phi (\nabla_VR)(X, Y) Z, U\}.\label{39d}\end{equation}
Using (\ref{24b}) we get from above
\begin{equation} g\{(\nabla_VR)(X, Y)\phi Z,U\} =-g\{ (\nabla_VR)(X, Y) Z, \phi U\}.\label{39e}\end{equation}
In view of (\ref{39e}) we obtain from (\ref{39b})
\begin{equation} (R(U,V).R)(X,Y)Z =[-g((\nabla_U R)(X,Y) V,\phi^2 Z)+ g((\nabla_V R)(X,Y)U,\phi^2 Z)]\xi, \label{39f} \end{equation}
which implies that
\begin{equation} (R(U,V).R)(X,Y)Z =[g((\nabla_U R)(X,Y) V, Z)- g((\nabla_V R)(X,Y)U,Z)]\xi, \label{39h} \end{equation}
i.e.
\begin{equation} (R(U,V).R)(X,Y)Z =[-(\nabla_U R)(X,Y,Z,V)+ (\nabla_V R)(X,Y,Z,U)]\xi ,\label{39i} \end{equation}
for any horizontal vector field $X$, $Y$, $Z$, $U$, $V$ on $M$. Hence we can state the following:
\begin{theorem} A Kenmotsu manifold $M$, $n\geq3$, is locally $\phi $-semisymmetric
if and only if the relation (\ref{39i}) holds for all horizontal vector fields $X$, $Y$, $Z$, $U$, $V$  on $M$.\end{theorem}
\section{ Characterization of  Locally $\phi $-semisymmetric Kenmotsu Manifolds}
In this section we investigate the condition of local $\phi $-semisymmetry of a  Kenmotsu manifold for arbitrary vector fields
on $M$. To find this we need the following results.
\begin{lemma} For any horizontal vector field $X$, $Y$ and $Z$ on a Kenmotsu manifold $M$, we have
\begin{equation} (\nabla_\xi R)(X,Y)Z=(\ell  _\xi R)(X,Y)Z+2R(X,Y)Z.\end{equation}\end{lemma}
\begin{proof} Let $X^\ast$, $Y^\ast$ and $Z^\ast$ be $\xi $- invariant horizontal vector field extensions on $X$, $Y$ and $Z$ respectively.
Since $X^\ast $ is $\xi $- invariant of $X$, we get by using (\ref{25}) 
\begin{equation} \nabla _\xi X^\ast =\nabla_X{^\ast}\xi  =X^\ast \label{41}\end{equation}
Now making use of invariance of $X^\ast$, $Y^\ast$ and $Z^\ast$ by $\xi $ and using (\ref{41}) we get
\begin{equation} \begin{array} {rcl} (\ell  _\xi R)(X^\ast ,Y^\ast )Z^\ast &=&[\xi ,R(X^\ast ,Y^\ast )Z^\ast ]\\
&=&\nabla_\xi (R(X^\ast ,Y^\ast )Z^\ast )-\nabla_{R(X^\ast ,Y^\ast )Z^\ast }\xi \\
&=&(\nabla_\xi R)(X^\ast ,Y^\ast )Z^\ast +R(\nabla_\xi X^\ast ,Y^\ast )Z^\ast +R(X^\ast ,\nabla_\xi Y^\ast )Z^\ast\\
&+&R(X^\ast ,Y^\ast )\nabla_\xi Z^\ast -R(X^\ast ,Y^\ast )Z^\ast \\
&=&(\nabla_\xi R)(X^\ast ,Y^\ast )Z^\ast +R( X^\ast ,Y^\ast )Z^\ast +R(X^\ast ,Y^\ast )Z^\ast\\
&+&R(X^\ast ,Y^\ast )Z^\ast -R(X^\ast ,Y^\ast )Z^\ast \\
&=&(\nabla_\xi R)(X^\ast ,Y^\ast )Z^\ast +2R( X^\ast ,Y^\ast )Z^\ast \end{array} \end{equation}
Hence we get the conclusion.\end{proof}
\begin{lemma} For any vector field $X$, $Y$ and $Z$ on a Kenmotsu manifold $M$ we have
\begin{equation} \begin{array} {rcl} R(\phi ^2X,\phi ^2Y)\phi ^2Z&=&-R(X,Y)Z+\eta (Z)\{\eta (X)Y-\eta (Y)X\}\\
&+&\{\eta (Y)g(X,Z)-\eta (X)g(Y,Z)\}\xi \end{array} \end{equation}\end{lemma}
Now lemma (4.1) and lemma (4.2) together imply the following:
\begin{lemma} For any vector field $X$, $Y$, $Z$ and $U$ on a Kenmotsu manifold $M$, we have
\be \begin {array} {rcl}&&(\nabla_{\phi ^2U} R)(\phi ^2X,\phi ^2Y)\phi ^2Z \\
&=& (\nabla_U R)(X,Y)Z-\eta (X)H_1(Y,U)Z+\eta (Y)H_1(X,U)Z+\eta (Z)H_1(X,Y)U\\
&+&\eta (U)[\eta (Z)\{\eta (X)\ell _\xi Y-\eta (Y)\ell _\xi X\}-(\ell _\xi R)(X,Y)Z]\\
&+&2\eta (U)[R(X,Y)Z-\eta (Z)\{\eta (X)Y-\eta (Y)X\}\\
&-&\{\eta (Y)g(X,Z)-\eta (X)g(Y,Z)\}\xi].\end{array} \label{42} \end{equation}
where the tensor field $H_1$ of type (1, 3) is given by
\begin{equation} H_1(X,Y)Z=R(X,Y)Z-g(X,Z)Y+g(Y,Z)X, \label{4.66} \end{equation}
for all vector fields $X$, $Y$, $Z$ on $M$.
\end{lemma}
Now let $X$, $Y$, $Z$, $U$, $V$ be arbitrary vector fields on $M$.\\
 Now we compute $(R(\phi ^2U, \phi ^2V).R)(\phi ^2X, \phi ^2Y)\phi ^2Z$
in two different ways. Firstly from (\ref{39i}), (\ref{21}) and (\ref{42}) we get
\begin{equation} \begin{array}{rcl} &&(R(\phi ^2U,\phi ^2V).R)(\phi ^2X,\phi ^2Y)\phi ^2Z\\
&=&\{(\nabla_U R)(X,Y,Z,V)-(\nabla_V R)(X,Y,Z,U)\}\xi\\
&+&\{\eta (U)\eta \{(\nabla_V R)(X,Y)Z)\}-\eta (V)\eta \{(\nabla_U R)(X,Y)Z)\}\}\xi\\
&-&\eta (X)\{H(Y,U,Z,V)-H(Y,V,Z,U)\}\xi \\
&+& \eta (Y)\{H(X,U,Z,V)-H(X,V,Z,U)\}\xi \\
&+&\eta (Z)\{H(X,Y,U,V)-H(X,Y,V,U)\}\xi\\
&+&\eta (X)\eta (Z)\{\eta(U)g(\ell _\xi Y,V)-\eta (V)g(\ell _\xi Y,U)\}\xi \\
&-&\eta (Y)\eta (Z)\{\eta (U)g(\ell _\xi X,V)-\eta (V)g(\ell _\xi X,U)\}\xi \\
&+&2\{\eta (U)R(X,Y,Z,V)-\eta (V)R(X,Y,Z,U)\}\xi \\
&+&2\eta (Z)\eta (V)\{\eta (X)g(Y,U)-\eta (Y)g(X,U)\}\xi \\
&-&2\eta (Z)\eta (U)\{\eta (X)g(Y,V)-\eta (Y)g(X,V)\}\xi , \end{array}\label{43} \end{equation}
where $H(X,Y,Z,U)=g(H_1(X,Y)Z,U)$ and the tensor field $H_1$ of type (1, 3) is given by (\ref{4.66})\\
 Secondly we have
\begin{equation} \begin{array}{rcl}&& (R(\phi ^2U,\phi ^2V).R)(\phi ^2X,\phi ^2Y)\phi ^2Z=R(\phi ^2U,\phi ^2V)R(\phi ^2X,\phi ^2Y)\phi ^2Z\\
&-&R(R(\phi ^2U,\phi ^2V)\phi ^2X,\phi ^2Y)\phi ^2Z-R(\phi ^2X,R(\phi ^2U,\phi ^2V)\phi ^2Y)\phi ^2Z\\
&-&R(\phi ^2X,\phi ^2Y)R(\phi ^2U,\phi ^2V)\phi ^2Z.\end{array}\label{44}  \end{equation}
By straightforward calculation from (\ref{44}) we get
\begin{equation} \begin{array}{rcl} &&(R(\phi ^2U,\phi ^2V).R)(\phi ^2X,\phi ^2Y)\phi ^2Z\\
&=&-(R(U,V).R)(X,Y)Z\\
&+&\eta (X)\{\eta (V)H_1(U,Y)Z-\eta (U)H_1(V,Y)Z\}\\
&+&\eta (Y)\{\eta (V)H_1(X,U)Z-\eta (U)H_1(X,V)Z\}\\
&+&\eta (Z)\{\eta (V)H_1(X,Y)U-\eta (U)H_1(X,Y)V\}\\
&+&\{\eta (V)g(H(X,Y,Z,U)-\eta (U)g(H(X,Y,Z,V)\}\xi ,
\end{array}\label{45}  \end{equation}
where $H(X,Y,Z,U)=g(H_1(X,Y)Z,U)$ and the tensor field $H_1$ of type (1, 3) is given by (\ref{4.66})\\
From (\ref{43}) and (\ref{45}) we obtain
\begin{equation} \begin{array}{rcl} && (R(U,V).R)(X,Y)Z\\
&=&[-(\nabla_U R)(X,Y,Z,V)+(\nabla_V R)(X,Y,Z,U)]\xi\\
&+&[\eta (V)\eta \{(\nabla_U R)(X,Y)Z)\}-\eta (U)\eta \{(\nabla_V R)(X,Y)Z)\}]\xi\\
&+&\eta (X)[\{H(Y,U,Z,V)-H(Y,V,Z,U)\}\xi+\eta (V)H_1(U,Y)Z-\eta (U)H_1(V,Y)Z]\\
&-& \eta (Y)[\{H(X,U,Z,V)-H(X,V,Z,U)\}\xi-\eta (V)H_1(X,U)Z+\eta (U)H_1(X,V)Z]\\
&-&\eta (Z)[\{H(X,Y,U,V)-H(X,Y,V,U)\}\xi+\eta (V)H_1(X,U)Z-\eta (U)H_1(X,V)Z]\\
&+&\{\eta (V)H(X,Y,Z,U)-\eta (U)H(X,Y,Z,V)\}\xi\\
&-&2\{\eta (U)R(X,Y,Z,V)-\eta (V)R(X,Y,Z,U)\}\xi \\
&+&.\{\eta (U)(\ell _\xi R)(X,Y,Z,V)-\eta (V)(\ell _\xi R)(X,Y,Z,U)\}\xi\\
&-&\eta (Z)\eta (X)\{\eta (U)g(\ell _\xi Y,V)-\eta (V)g(\ell _\xi Y,U)\}\xi \\
&+&\eta (Z)\eta (Y)\{\eta (U)g(\ell _\xi X,V)-\eta (V)g(\ell _\xi X,U)\}\xi \\
&-&2\eta (Z)\eta (V)\{\eta (X)g(Y,U)-\eta (Y)g(X,U)\}\xi \\
&+&2\eta (Z)\eta (U)\{\eta (X)g(Y,V)-\eta (Y)g(X,V)\}\xi . \end{array}\label{46} \end{equation}
Thus in a locally $\phi $-semisymmetric Kenmotsu manifold the relation (\ref{46}) holds for all arbitrary vector fields
$X$, $Y$, $Z$, $U$, $V$ on $M$. Next if the relation (\ref{46}) holds in a Kenmotsu manifold, then for any horizontal vector field
$X$, $Y$, $Z$, $U$, $V$ on $M$, we get the relation (\ref{39i}) and hence the manifold is locally $\phi $-semisymmetric. \\
Thus we can state the following:
\begin{theorem} A Kenmotsu manifold $M$ is locally $\phi $-semisymmetric if and only if the relation (\ref{46}) holds for any
arbitrary vector field $X$, $Y$, $Z$, $U$, $V$ on $M$.\end{theorem}

\hspace{10pt}

\hspace{10pt}
\end{document}